\documentclass[a4paper,reqno,12pt,oneside]{amsart}
\usepackage{amsmath,geometry}
\usepackage{graphicx}
\usepackage{mathrsfs}
\usepackage{amssymb,amsmath, amsfonts, amsthm}
\usepackage{latexsym}
\usepackage{color} 
\usepackage[dvips]{epsfig}
\usepackage[colorlinks]{hyperref}
\usepackage{bm}
\newtheorem{theorem}{Theorem}[section]
\newtheorem{lemma}[theorem]{Lemma}

\numberwithin{equation}{section}

\renewcommand{\(}{\left(}
\renewcommand{\)}{\right)}
\newcommand{\beq}{\begin{equation}}
	\newcommand{\eeq}{\end{equation}}
\newcommand{\ba}{\begin{aligned}}
	\newcommand{\ea}{\end{aligned}}

\definecolor{darkblue}{rgb}{0.05, .05, .65}
\definecolor{darkgreen}{rgb}{0.1, .65, .1}
\definecolor{darkred}{rgb}{0.8,0,0}

\theoremstyle{plain}
\newtheorem{thm}{Theorem}[section]
\newtheorem{lem}[thm]{Lemma}
\newtheorem{prop}[thm]{Proposition}

\numberwithin{equation}{section}
\theoremstyle{remark}

\DeclareMathOperator{\dist}{dist}

\DeclareMathOperator{\vspan}{span}

\renewcommand{\(}{\left(}
\renewcommand{\)}{\right)}

\renewcommand{\O}{\mathcal O}

\newcommand{\eps}{\varepsilon}
\newcommand{\N}{\mathcal N_{\de,\xi}}
\newcommand{\E}{\mathcal E_{\de,\xi}} 
\newcommand{\Om}{\Omega}

\renewcommand{\r}{\rangle}
\renewcommand{\l}{\langle}
\renewcommand{\leq}{\leqslant}

\newcommand{\RR}{{\mathbb{R}^4}}

\renewcommand{\L}{\mathcal{L}}

\newcommand{\U}{U_{\delta,\xi}}

\newcommand{\de}{\delta}

\begin{document}
	\title[The Brezis-Nirenberg problem in 4D]{The Brezis-Nirenberg problem in 4D}

 \author{Angela Pistoia}
\address{Angela Pistoia, Dipartimento di Scienze di Base e Applicate per l'Ingegneria, Sapienza Universit\`a di Roma, Via Antonio Scarpa 10, 00161 Roma (Italy)}
\email{angela.pistoia@uniroma1.it}

 \author{Serena Rocci}
\address{Serena Rocci, Dipartimento di Scienze di Base e Applicate per l'Ingegneria, Sapienza Universit\`a di Roma, Via Antonio Scarpa 10, 00161 Roma (Italy)}
\email{serena.rocci@uniroma1.it}

\begin{abstract}
We address the existence of blowing-up solutions for the Brezis-Nirenberg problem in 4D.\\

\centerline{\em Dedicated to Yihong Du's 60th birthday}
\end{abstract}
\date\today
\thanks{ The authors are partially supported by the group GNAMPA of the Istituto Nazionale di Alta Matematica (INdAM)
}
\maketitle
	\section{Introduction}
 The   problem
 \begin{equation}
 \label{bn}
 -\Delta u=|u|^{4\over n-2}u+\lambda V u\ \hbox{in}\ \Omega,\ u=0\ \hbox{on}\ \partial\Omega
 \end{equation}
 where $\Omega$ is a bounded regular domain in $\mathbb R^n$, $\lambda\in \mathbb R$ and $V\in C^0(\overline \Omega)$
 was  introduced by Brezis and Nirenberg in the pivotal paper \cite{bn}, where they address the  existence of positive solutions 
 in the autonomous case, i.e.   the potential $V$ is constant.  Since then, a huge aumont  of work has been done. 
 In the following we will make a brief history highlighting the results which are much closer to the problem we wish to study in the present paper.
 \\
 
 It is well known that existence and multiplicity of positive and sign-changing solutions to \eqref{bn} is strictly affected by the geometry of the domain, the dimension of the euclidean space where the domain lies and the values of the parameter $\lambda$.
For example, if $\lambda\le 0,$ $V$  is constant and $\Omega$ is a starshaped domain, problem \eqref{bn} does not have any solutions.
\\
A particular feature of problem \eqref{bn}, due to the critical behaviour of the non-linearity which appears on the R.H.S., is the  possible existence of solutions which  blow-up at one or more points in the domain as the parameter $\lambda$ approaches $0$ in dimensione $n\ge4$ or  a positive number $\lambda_*$ in dimension $n=3.$ The description of the profile of the positive blowing-up solutions has been the subject of a wide literature.
Starting from the pioneering paper by Brezis and Peletier \cite{bre-pe} where the authors consider the radial case, a lot of results have been obtained concerning the asymptotic profiles of solutions to \eqref{bn}. When $n\ge4$, the first result is due to Han \cite{han} and Rey \cite{rey1} who study the profile of one-peak solutions in a general domain when the potential V is constant, while
 the most recent one  has been obtained by Konig and Laurin \cite{ko-lau1} who give an accurate description of the blow-up profile of a  solution with multiple blow-up points in the non-autonomous case.
The 3-dimensional case has been firstly faced by Druet \cite{dru} in the case of one blow-up point and very recently by Konig and Laurain \cite{ko-lau2} in the presence of multiple blow-up points.\\
In particular,   the asymptotic analysys of blowing-up solutions   ensures that the blow-up points are nothing but  the critical points of a suitable function which involves the function V, the  Green's function of $-\Delta$ in $\Omega$ with Dirichlet boundary condition and the Robin's function.\\
While the behaviour of blowing-up positive solutions is by now clear, the  profile of the sign-changing solutions is far being completely understood. As far as we know the complete scenario of sign-changing solutions  has been obtained only in the radial case by Esposito, Ghossoub, Pistoia and Vaira in \cite{es-gho-pi-va} in dimensions $n\ge7$ and by Amadori, Gladiali, Grossi, Pistoia and Varia \cite{aggpv}
  in lower dimensions.
 \\

A  parallel and prolific field of research consists in finding blowing-up  solutions whose profile is predicted by the asymptotic analysis developed in the previous papers. The first result is due to Rey \cite{rey1}
who build solutions blowing-up at a non-degenerate critical point of the Robin function when $V$ is constant and $n\ge5.$ Successively, positive solutions blowing-up at multiple points   have been constructed by Musso and Pistoia  \cite{mu-pi}.  If $n\ge5$ in the  non-autonomous case, solutions with one blow-up point  have been found by Pistoia and Molle \cite{mo-pi}, while  Micheletti and Pistoia  \cite{mi-pi} build positive and sign-changing solutions with multiple blow-up points. \\
While all the blow-up points of positive solutions are always  isolated and simple, in high dimensions $n\ge7$ sign-changing solutions  can have
 towering blow-up points 
(see Premoselli \cite{pre2} and Morabito-Pistoia-Vaira \cite{mo-pi-va})  and clustering blow-up points (see Pistoia Vaira \cite{pi-va1}). \\
In the 3-dimensional case, when $V$ is constant,  single and multiple blowing-up  solutions  have been built  by del Pino, Dolbeault and Musso \cite{de-do-mu} and Musso and Salazar 
\cite{mu-sa}, respectively.
\\

Even if the problem has been widely studied in the last decades and a huge number of results has been produced, many problems are still open.

Here we focus on the problem \eqref{bn} in dimension $n=4$: 
	\begin{equation}
		\begin{cases}
			-\Delta u   = u^3 + \eps V u &\mbox{in } \Omega\\
			u=0	&\mbox{on } \partial \Omega
		\end{cases}
		\label{prob}
	\end{equation}
	 when $\Omega$ is a bounded regular domain in $\RR$, $V\in C^1(\Omega)\cap C^0(\overline\Omega)$ 
 and we  prove the existence of solutions which blows-up at a single point as the parameter $\epsilon$ approaches zero.
\\
More precisely, if $\tau$ denotes the Robin function, our main result reads as follows.

\begin{thm}\label{main}
Let ${\xi}_0\in\Omega$ be a non-degenerate critical point of 
$f(\xi):={\tau(\xi)\over V(\xi)}$
with $V(\xi_0)>0.$\\
If $\epsilon$ is small enough there exists a solution of problem
 \eqref{prob} which blows-up  at the point $\xi_0$ as $\eps \to 0 $.
\end{thm}

The proof relies on a classical Ljapunov-Schmidt procedure. The setting of the problem (see Section \ref{2}) and the reduction process
(see Section \ref{3}) can be carried out as usual. However, the last step in the procedure   needs  new ideas. 
 In fact, the rate of the error term  is not small enough to argue as in the   higher dimensional case and 
   the reduced problem is solved using some
local Pohozaev identities (see Section \ref{4}).
\\

Remarkably, differently from solutions with one blow-up point,  in the  case of multiple blow-up points it is  harder to derive the concentration speed.
Indeed, it seems that they appear at the second order expansion   and so to catch them a more accurate description of the ansatz is needed.
This will be the topic of a forthcoming paper in collaboration with Monica Musso.

	\section{Setting of the problem}\label{2}
	\subsection{The bubbles}
		
All the positive solutions to the limit problem
	\begin{equation*}
		-\Delta U = U^3 \mbox{ in } \RR
		\label{limprob}
	\end{equation*}
	are the so called {\em bubbles} (see \cite{au,ta})
	\begin{equation*}\label{bu1}
 	U_{\delta, \xi} (x) = \frac1\delta U\(x-\xi\over\delta\),\  x,\xi \in \mathbb{R}^4,\  \delta>0	\end{equation*}
 	where
\begin{equation*}\label{bu2}
 U(x):= \mathfrak c \frac{1}{1 + |x |^2},\ \mathfrak c:=2\sqrt{2}.\end{equation*}

	It is useful to introduce the projection of the bubble $U_{\delta,\xi}$ onto $H^1_0(\Omega)$, namely the solution of the problem
$$
		-\Delta P U_{\delta,\xi} =-\Delta U_{\delta,\xi}\ \mbox{ in } \Omega,\ 		PU_{\delta,\xi}=0\ \mbox{ on } \partial \Omega
.$$
	
	Let  $G$ be the Green function of $-\Delta$ on $H_0^1(\Om)$ and $H$ be its regular part, i.e.
	$$ G(x,y) =\frac1{2\omega}\frac{1}{|x-y|^2}-H(x,y) $$
	where    $\omega $ denotes the measure of the unit sphere $S^3 \subset \RR$.
	The Robin function  is defined as $\tau(x) = H(x,x) .$\\

It is well known that
	\begin{equation*}\label{pud1}
PU_{\delta,\xi}(x)=U_{\delta,\xi}(x)-  \mathfrak C \de H(x, \xi)+\mathcal O(\delta^3),\  \mathfrak C :=2\mathfrak c\omega\end{equation*}
  uniformly with respect to $x\in\Omega$ and $\xi$ in compact sets of $\Omega$
  and
\begin{equation*}\label{pud2}
		 PU_{\delta, \xi}(x)= \mathfrak C \de G(x,\xi)+ \mathcal O(\delta^3)\end{equation*}
  uniformly  with respect to  $x$ in compact sets of $\Omega\setminus\{\xi\}$ and $\xi$ in compact sets of $\Omega$
	\subsection{Some background material}
	Let $H^1_0(\Omega)$ be the Hilbert space  equipped with the usual inner product and the usual norm $$		\l u,v\r= \int_{\Omega} \nabla u \cdot \nabla v
\ \hbox{and}\
		\|u\|:=\|u\|_{H^1_0(\Om)} = \left( \int_{\Omega} |\nabla u|^2 \right)^{1/2}.$$
	For $r \in [1,+\infty)$ the space $  L^r(\Om)$ is also equipped with the standard norm
	\begin{equation*}
		\|u\|_r=\left(\int_{\Omega} |u|^r\right)^{\frac{1}{r}} .
	\end{equation*}
	
Now, let us introduce ${\mathtt i}^*:L^{4/3}(\Om) \to H^1_0(\Om)$ as the adjoint operator of the embedding $\mathtt i: H^1_0(\Om) \hookrightarrow L^4(\Om)$, i.e.
		 $ u={\mathtt i}^*(f) $ if and only if $$\l u,\varphi\r = \int_\Om f(x) \varphi(x) d x \mbox{ for all } \varphi\in H^1_0(\Om) $$ or equivalently  
	$$	-\Delta u=f  \mbox{ in } \Om,\  u=0  \mbox{ on } \partial\Om$$
The operator ${\mathtt i}^*:L^{4/3}(\Om)\to H^1_0(\Om)$ is continuous as $$\|{\mathtt i}^*(f)\|_{H^1_0(\Om)}\leq S^{-1 }\|f\|_{4/3}$$ where $S$ is the best constant for the Sobolev embedding.
\\

Therefore,
  the problem \eqref{prob} can be rewritten as \beq\label{eq2}
		u={\mathtt i}^*(u^3+\eps V u),  \quad u \in H^1_0(\Omega) .\eeq  
	
\subsection{The ansatz}
 		Let   $k \geq 1$ be a fixed integer.  We look for a solution to \eqref{prob} of the form 
		\beq \label{ans} u=U_{\de,\xi}+\phi, \eeq where $U_{\de,\xi}=   PU_{\de,\xi} $ whose
	  blow-up rate is $\de=\de(\eps) \to 0$ and  blow-up point is $\xi = \xi(\eps) \in \Om$. 
		Moreover, the lower order term 
	$ 	\phi $    satisfies a set of orthogonal conditions. More precisely, let us consider the linear problem
	$$-\Delta\psi=3U ^2\psi \mbox{ in } \RR.$$ 
	It is well known \cite{bia-eg} that the set of solutions  is a $5-$dimensional linear space spanned by 	
$$ \psi^0 (x):= U(x)+\frac 12 \nabla U(x)\cdot x=-\mathfrak c\frac{1-|x |^2 }{(1+|x |^2)^2} $$ and 
$$  \psi ^j(x):= \frac{\partial U}{\partial x_j}(x)=-2\mathfrak c  \frac{x_j}{(1+|x |^2)^2}, \ j=1,\dots,4.$$
	For $j=0,1,2,3,4$ we set 
	$$\psi_{\de,\xi}^j(x)=\frac1{\de}\psi^j\(x-\xi\over\delta\).$$
	We introduce their projections $P\psi_{\de,\xi}^j={\mathtt i}^*(3\U^{2}\psi_{\de,\xi}^j)$ onto $H^1_0(\Omega),$ namely the solutions to the problem
	$$
		-\Delta P\psi_{\de,\xi}^j =-\Delta \psi_{\de,\xi}^j=3U_{\de,\xi} ^2 \psi_{\de,\xi}^j\ \hbox{in}\ \Omega,\ 		P\psi_{\de,\xi}^j =0\ \hbox{in}\  \partial \Omega
.$$
	Finally, we introduce the linear space
	$$K_{{\de},{\xi}}=  \vspan \{ P\psi_{\de,\xi}^j\ |\  j=0,\cdots,4\}$$ 
	and its orthogonal space $$K_{\de,\xi}^\perp=\{\phi\in H_0^1(\Om) : \l\phi,P\psi_{\de,\xi}^j\r=0, \mbox{ for all }  j=0,\cdots,4\} .$$  
The function $\phi$ belongs to the space $ K_{\de,\xi}^\perp.$

It is useful to remind the well known properties
$$P\psi^j_{\delta, \xi}=\psi^j_{\delta,\xi}-\delta^{2}\mathfrak C\partial_{\xi_j}H(x,\xi)+\mathcal O(\delta^{3}),\quad j=1, \ldots, 4$$
uniformly for $x\in\Omega$ and $\xi$ in compact	sets of $\Omega$
and
	$$P\psi^0_{\delta, \xi}=\psi^0_{\delta,\xi}-\delta \mathfrak C H(x,\xi)+\mathcal O(\delta^{2})$$ 
	uniformly for $x $  in compact	sets of $\Omega\setminus\{\xi\}$ and $\xi$ in compact	sets of $\Omega.$

\subsection{An equivalent system}
Let us introduce the linear projection
  $ \Pi_{{\de},{\xi}}:K\to K$ and $ \Pi_{{\de},{\xi}}^\perp:K^\perp\to K^\perp$ which are defined by $$\Pi_{{\de},{\xi}}(\phi)=\sum_{j=0,\cdots,4}\l\phi,P \psi_{\delta,\xi}^j\r P \psi_{\delta,\xi}^j \qquad \hbox{and}\qquad \Pi_{{\de},{\xi}}^\perp(\phi) =\phi-\Pi_{{\de},{\xi}}(\phi) .$$
Equation \eqref{eq2} can be rewritten as the following system of two equations
	\beq 	\label{equ1}
	\Pi_{{\de},{\xi}}^\perp\big[\L_{\de,\xi}(\phi) -\E-\N(\phi)\big] =0\eeq 
	and
	\beq 	\label{equ2}
	\Pi_{{\de},{\xi}} \big[\L_{\de,\xi}(\phi) -\E-\N(\phi)\big] =0,\eeq 
	where the linear operator $\L_{\de,\xi}$ is  
$$\L_{\de,\xi} (\phi)=
	 \phi -{\mathtt i}^*\left(3\phi W_{{\de},{\xi}}^2+\eps\phi\right) ,$$
		the error term $\E$ is 
$$	\E=  {\mathtt i}^*\left(W_{{\de},{\xi}}^3  +\eps W_{{\de},{\xi}}\right)-W_{{\de},{\xi}} $$
and the nonlinear term 	$\N$ is
$$\N(\phi) = {\mathtt i}^*\left(\phi^3+3\phi^2 W_{{\de},{\xi}}\right) .$$

\section{Solving equation (\ref{equ1})}\label{3}
Given $\rho>0$   small,  let
$	\O_\rho= \left\{ {\xi} \in \Om\  | \ \dist(\xi,\partial\Om)\geq \rho \right\}. $	

First of all, we estimate the error term $\E.$  
	\begin{lem} \label{err1 lem} For any $\rho >0$ small enough there exist $\eps_0>0$ and $c>0$ such that for any $\xi \in \O_\rho$ and for any $\eps,\delta\in (0,\eps_0)$ it holds
	$$
 		\|\E\| \lesssim 
		| {\de} |^2  + \eps | {\de} |  .$$
 	\end{lem}
 	\begin{proof}  
				First of all, we remark that
				$W_{{\de},{\xi}}={\mathtt i}^*\left(  \U^3\right).$ Then by a straightforward computation  
		\begin{align*}
			\|\E\| &\lesssim   \| (  P\U)^3- 
			\U^3\|_{\frac 43} + \eps \|V\|_\infty\|P\U\|_{\frac 43}  	\end{align*}
	where 
	\begin{align*}
		\|P\U^3-\U^3\| &\lesssim \de  \|\U^2\|_{\frac 43} + \de^2\|\U\|_{\frac 43} + \de^3\lesssim \de^2 .
	\end{align*}
and  $\|P\U\|_{\frac 43}= \O(\de)$.	\end{proof}
Next, we state the invertibility of the linear operator $\L_{\de,\xi}.$
	\begin{lem}\label{L inv prop}
		For any $\rho>0$ small enough there exist $\eps_0>0$ and $c>0$ such that for any $\eps,\delta\in (0,\eps_0),$ $i=1,\cdots,k$, and for any ${\xi} \in\O_\rho$
			$$\|\(\Pi^\perp_{\de,\xi} \circ \L_{\de,\xi}\)(\phi)\|\geq c\|\phi\| \mbox{ for all } \phi \in K_{\de,\xi}^\perp .$$
		Furthermore, the operator $\Pi^\perp_{\de,\xi} \circ \L_{\de,\xi}$ is invertible and its inverse is continuous.
	\end{lem}
	\begin{proof} We omit the proof because  it is enough to apply the arguments used in the proof of Lemma 1.7 in \cite{mu-pi} to the $4-$dimensional case.
	\end{proof}
	
	Finally, using a classical fixed point argument, we can solve equation \eqref{equ1}.

\begin{prop} \label{phi norm prop}For any  $\rho>0$ small enough there exists $\eps_0>0$  such that
for any $\eps,\delta \in (0,\eps_0)$  and $\xi \in \O_\rho$, there exists a unique $\phi = \phi_{\de,\xi}\in K_{\de,\xi}^\perp$ solving
\eqref{equ1}, which also satisfies
\beq \|\phi\|\lesssim  {\de}^2 +\eps {\de}  . \label{phi norm}\eeq
  \end{prop}

\section{Solving equation (\ref{equ2})}\label{4}
 Let $u=U_{\de,\xi}+\phi$ (see \eqref{ans}). By  \eqref{equ1} we deduce that
 \begin{equation}\label{cij}-\Delta u-u^3-\epsilon Vu=\sum\limits_{j=0,\dots,4} c^j \U^2\psi_{\de,\xi}^j\end{equation}
for some real numbers $c^j= c^j(\de,\xi)$. Then
solving equation \eqref{equ2} is equivalent to find $(\de,\xi)$ such that the $c^j$'s are zero.\\

First of all, we prove a sufficient condition which ensures that all the $c^j$'s in \eqref{cij} are zero.
Set $\partial_j u:={\partial u\over\partial x_j}.$

\begin{lemma}\label{lemma1}
If 
\begin{equation}\label{ridotto1}
\int\limits_{\Omega}(-\Delta u-u^3-\epsilon V u)\psi_{\de,\xi}^0 =0
\end{equation}
and for some $\eta>0$
\begin{equation}\label{ridotto2}
\int\limits_{B(\xi,\eta)}(-\Delta u-u^3-\epsilon Vu)\partial_j u=0,\quad \mbox{ for all } j=1,2,3,4\end{equation}
then  $c^\ell=0$ for any $\ell=0,1,\dots,4.$

\end{lemma}
\begin{proof} By \eqref{ridotto1} and \eqref{ridotto2}  
$$\sum\limits_{\ell=0}^4 c^\ell \int\limits_\Omega \U^2\psi_{\de,\xi}^\ell\psi_{\de,\xi}^0= \sum\limits_{\ell=0}^4 c^\ell \int\limits_{B(\xi,\eta)} \U^2 \psi_{\de,\xi}^\ell\partial_j u =0 \quad \mbox{ for all } j=0,\dots,4$$
and the linear system in the $c^\ell$'s is diagonally dominant. Indeed
$$\int\limits_\Omega U_{\delta,\xi} ^2 \psi_{\delta,\xi} ^\ell \psi_{\delta,\xi}^0 =  \begin{cases}
	a+ \O(\delta^4) 		&\mbox{if } \ell =0\\
	\O(\delta^5)		&\mbox{if }\ell=1,\cdots,4\\
\end{cases}$$
for some constant $a\not=0.$
Moreover, 
$$\int\limits_{B(\xi,\eta)} \U^2 \psi_{\de,\xi}^\ell\partial_j u =  \begin{cases}
  \frac{1}{\de} b +\O(\eps)	&\mbox{if } \ell =j\\
	\O(\eps)		&\mbox{if } \ell\neq j\\
\end{cases}$$
for some constant $b\not=0,$ because by  \eqref{ans} and  \eqref{phi norm}, 
$$ \int\limits_{B(\xi_i,\eta_i)} U_{\delta,\xi} ^2 \psi_{\delta,\xi} ^\ell\partial_j\phi 
= 	\O(\|\phi\|) \(\int_{B(\xi,\eta)} |\U\psi_{\de,\xi}^\ell|^2\)^{1/2} = \O(\eps)	\quad \mbox{for all } \ell=0,\cdots, 4
$$
and
$$ \int\limits_{B(\xi,\eta)} \U^2 \psi_{\de,\xi}^\ell\partial_jP\U
= \begin{cases}
  \frac{1}{\de} b +\O(\de^2) &\mbox{if }\ell=j\\
	\O(\de^2)		&\mbox{if } \ell\neq j \\
		\end{cases} 
$$

\end{proof}

Next, we write the first order term of \eqref{ridotto1}.	

\begin{lemma}\label{lemma2} For any $\rho>0$ small enough
	\begin{equation}\label{pro0} \int\limits_{\Omega}(-\Delta u-u^3-\epsilon V(x)u) \psi_{\de,\xi}^0= \mathfrak C^2  \de^2 \tau(\xi) +\eps \de^2 \ln\de \mathfrak c^2\omega V(\xi)  +  o(\de^2) , \end{equation} 
	 as $\eps,\de\to 0$, uniformly respect to $\xi\in \O_\rho$.
\end{lemma}

\begin{proof}
We point out that
$$\begin{aligned}
\int\limits_{\Omega}(-\Delta u-u^3-\epsilon Vu) \psi_{\delta,\xi}^0&=\underbrace{\int\limits_{\Omega}\(-\Delta U_{\de,\xi}-U_{\de,\xi}^3-\epsilon U_{\de,\xi}\) \psi_{\delta,\xi}^0}_{=:I_1}\\ &
+\underbrace{\int\limits_{\Omega}\(-\Delta \phi-3U_{\de,\xi}^2\phi\) \psi_{\delta,\xi}^0}_{=:I_2}\\
&+\underbrace{\int\limits_{\Omega}\( \phi^3+3U_{\de,\xi} \phi^2\) \psi_{\delta,\xi}^0}_{=:I_3}\
\end{aligned}$$
First of all, let us prove that   
\beq\label{pro01}I_1 = \de^2 \mathfrak C^2H(\xi,\xi)+\eps \de^2 \ln\de \mathfrak c^2 \omega  V(\xi)  + o(\de^2) .\eeq
		We observe that  \begin{align*}I_1 &= \int_\Om \left[   \U^3 -  P\U^3 - \eps V(x)    P\U \right]   \psi_{\de,\xi}^0 \\&= \int_\Om (\U^3-P\U^3)  \psi_{\delta,\xi}^0 - \eps \int_\Om V(x)  P\U \psi_{\de,\xi}^0
		.  \end{align*} 
	Now
	\begin{align*}
		\int_\Om ( \U^3-P\U^3)&\psi_{\de,\xi}^0 =  3 \int_{B(\xi,\rho)} \U^2 \underbrace{(\U-P\U)}_{\de  \mathfrak C H(x,\xi) +\O(\de^3)} \psi_{\de,\xi}^0 + \O(\de^3) \\& 
		=3\de  \mathfrak C \int_\Om  H(x,\xi)\U^2\psi_{\de,\xi}^0+ \O(\de^3)
		\\&= 3\de^4  \mathfrak C \mathfrak c^3 \int_{B(\xi,\rho)}  H(x,\xi) \frac{|x-\xi|^2-\de^2}{\left(\de^2+|x-\xi|^2\right)^4}  + o(\de^2)\\
		&=3 \de^2  \mathfrak C \mathfrak c^3 \int_{B(0,\rho/\de)} \underbrace{H(\xi+\de t,\xi)}_{H(\xi,\xi)+\O(\de)} \frac{|t|^2-1}{(1+|t|^2)^4} + o(\de^2)  \\
		&= \de^2  \mathfrak C^2 H(\xi,\xi) + o(\de^2)
	\end{align*}
	because
		\begin{equation*}
  \int_{\mathbb R^4}   \frac{|t|^2-1}{(1+|t|^2)^4}  =  \frac\omega{12},	\end{equation*}
and 
\begin{align*} \eps \int_\Om V(x) P\U \psi_{\de,\xi}^0  &= \eps \de^2 \mathfrak c^2 \int_{B(\xi,\rho)} V(x)\frac{|x-\xi|^2-\de^2}{(\de^2+|x-\xi|^2)^3} +\O(\de^3) \\&= \eps \de^2 \mathfrak c^2 \int_{B(0,\rho/\de)} V(\de t+ \xi)\frac{|y|^2-1}{(1+|y|^2)^3} +\O(\de^3)\\ & = -\eps \de^2 \ln\de \mathfrak c^2 \omega V(\xi) + o(\de^2 \eps \ln\de) .
\end{align*}
Let us estimate the other terms.
It is important to point out the estimate $$
 \int_{\partial\Om} |\partial_{\nu}\phi|^2 = o(\de^2)
 $$   proved in \cite{rey1}. 
 Then, recalling that $-\Delta  \psi_{\delta,\xi}^0=3U_i^2 \psi_{\delta,\xi}^0$, for all $i=1\cdots, k$,   we have
\begin{align*}
	\int_\Om (-\Delta\phi)  \psi_{\delta,\xi}^0 &= \underbrace{\int_\Om \phi (-\Delta \psi_{\delta,\xi}^0)}_{=3\int_\Om \phi U_i^2  \psi_{\delta,\xi}^0} + \int_{\partial \Om} \underbrace{\phi}_{=0} \nabla  \psi_{\delta,\xi}^0\cdot \nu  -\int_{\partial\Om } \underbrace{ \psi_{\delta,\xi}^0}_{=\O(\de)}\nabla \phi\cdot\nu \\&= 3\int_\Om \phi U_i^2 \psi_{\delta,\xi}^0 +o(|\de|^2) .
\end{align*}
Therefore,
\begin{equation} \label{I2w} \ba
	 |I_2|&\lesssim
	  3\int_\Om \phi |P\U^2-\U^2|| \psi_{\delta,\xi}^0| +\eps  \int_\Om |\phi||\psi_{\de,\xi}^0|+o(|\de|^2) \\&\lesssim
	   \|\phi\|_4 \|(P\U^2-U_i^2) \psi_{\delta,\xi}^0\|_{4/3} + \eps \|\phi\|_4 \| \psi_{\delta,\xi}^0\|_{4/3}=o\(\de^2\) . \ea
\end{equation}
and \begin{equation} \label{I3}\ba
		|I_3| &\leq  \int_\Om  \left[|\phi|^3+3|\phi|^2|   P\U|\right] |\psi_{\de,\xi}^0| 	\\ &\lesssim \|\phi\|^3 \| \psi_{\delta,\xi}^0\|_4+\|\phi\|^2
		\| P\U  \psi_{\delta,\xi}^0 \|_2= o (\de^2).\ea\end{equation}
	Finally, \eqref{pro0} follows by \eqref{pro01}, \eqref{I2w} and \eqref{I3}.
\end{proof}

Finally, we write the first order term of \eqref{ridotto2}.	
  
\begin{lemma} \label{lemma3} For any $\rho>0,$ there exists $\eta_i>0$ such that  
$$ \int\limits_{B(\xi_i,\eta_i)}(-\Delta u-u^3-\epsilon Vu)\partial_j u=
 -\frac 12  \de^2 \left[ \mathfrak C^2  {\partial_j}\tau(\xi) +(\eps\ln\de) \omega \mathfrak c^2 {\partial_j}V(\xi)  \right]+o\(\delta^2\), $$
	 as $\eps,\de\to 0$, uniformly respect to $\xi\in \O_\rho$.
\end{lemma}  
\begin{proof}
First of all, we point out that
	\begin{equation} \label{key1}
		\int\limits_{B(\xi_i,\eta_i)}(-\Delta u-u^3
		)\partial_j u =		\int\limits_{\partial B(\xi_i,\eta_i)}
		\left(-\partial_\nu u\partial_j u+\frac12 |\nabla u|^2\nu_j-\frac14 u^4\nu_j
		\right)
	\end{equation}
	By Lemma \ref{coarea lem} and \eqref{phi norm}, we choose $\eta_{i}$ such that
\begin{equation} \label{key2}  
  \int_{\partial B(\xi_i,\eta_{i})}\left( |\nabla \phi|^2 + |\phi|^4
  \right) \lesssim \de^4.  \end{equation}

Now, by \eqref{ans}  
	\begin{equation} \label{key3}  
	U_{\de,\xi}(x):= {\de\mathfrak c\over |x-\xi|^2}-    \mathfrak C\delta H(x,\xi)+\mathcal O\(|\de|^2\)\end{equation}
	$C^1-$uniformly on $\partial B(\xi,\eta).$  
It is crucial to point out that the function $H(x,\xi)$ is harmonic. 
Therefore, by \eqref{key1}, \eqref{key2} and \eqref{key3} 
$$\begin{aligned}
 &	\int\limits_{B(\xi,\eta)}(-\Delta u-u^3
 )\partial_j u \\ =&	\de^2 \int\limits_{\partial B(\xi,\eta)}
		\left[-\partial_\nu \( { \mathfrak c\over |x-\xi|^2}+   \mathfrak C H(x,\xi)\)\partial_j \( {\mathfrak c\over |x-\xi|^2}+   \mathfrak CH(x,\xi)\)\right. \\&+ \de^2 \frac12 \left. \left|\nabla \( {\mathfrak c\over |x-\xi_i|^2}+   \mathfrak CH(x,\xi)\)\right|^2\nu_j \right]+o\(|\de|^2\)\\
		=& -\de^2\int_{\partial B(\xi,\eta)} \nabla \(\mathfrak c \over |x-\xi|^2\) \cdot \nu \partial_{x_j}H(x,\xi) + o(\de^2)
		\\=&	2\mathfrak c \mathfrak C\delta^2\frac1{|\eta|^3}\int\limits_{\partial B(\xi,\eta)}
   \partial_j H(x,\xi) +o\(\de^3\)={\mathfrak C^2}\delta^2\partial_jH(x,\xi)\big|_{x=\xi}+o\(\de^2\)
\end{aligned}$$
because $H(x,\xi)$ is harmonic  and also
 $2\mathfrak c \omega =  \mathfrak C$ .\\
   Finally, as $ \tau(x) = H(x,x) $ denotes the Robin's function, by
  $$ \partial_{\xi_j}\tau(\xi)= (\partial_{x_j} H(x,y) + \partial_{y_j}H(x,y))|_{(x,y)=(\xi,\xi)} = 2\partial_{x_j}H(x,y)|_{(x,y)=(\xi,\xi)}$$
    follows that 
    \begin{equation}\label{key14}
    	\int_{B(\xi,\eta)} (-\Delta u -u^3)\partial_j u = \frac 12 \mathfrak C^2 \de^2 \partial_j \tau(\xi)+o(\de^2). 
    \end{equation}

Arguing in a similar way, we also have
\begin{equation} \label{key15}\ba
-\eps \int_{B(\xi,\eta)} V(x) u \partial_j u &=\eps \frac12 \int_{B(\xi,\eta)}  \(\partial_j V(x)\) u^2-\eps \frac12 \int_{\partial B(\xi,\eta)}   V(x)  u^2\nu_j\\
&=-\frac 12\eps \de^2 \omega \mathfrak c^2 \partial_j V(\xi) \ln\de + o(\de^2) 
\ea\end{equation}

By \eqref{key14} and \eqref{key15} the claim follows.
    \end{proof}
  
    \begin{lemma} \label{coarea lem}
 If there exists $C_1>0$ such that \beq \label{coarea1} \int_{B(\xi,\eta_1)\backslash B(\xi,\eta_2)} |f(x)| dx \leq C_1, \eeq then there exist $C_2>0$ and $\bar{\eta}\in(\eta_1,\eta_2)$ such that $$\int_{\partial B(\xi,\bar{\eta})} |f (x)|dx  \leq C_2 .$$
\end{lemma}	\begin{proof} Assume that for any $\eta \in (\eta_1,\eta_2)$ and $C>0$
 \[ \int_{\partial B(\xi,\eta)} |f(x)|dx > C .\] Then by the coarea formula 
  \[ \int_{B(\xi,\eta_1)\backslash B(\xi,\eta_2)}|f(x)|dx = \int_{\eta_1}^{\eta_2} \left(\int_{\partial B(\xi,\eta)} |f(x)|dx \right) d\eta > C(\eta_2-\eta_1),  \] and a contradiction arises.
	\end{proof}

    \subsection{Proof of Theorem \ref{main}: completed}
If we set $\de=e^{-\frac{t}{\eps}}$, with $t>0$, by Lemma \ref{lemma1}, Lemma \ref{lemma2} and Lemma \ref{lemma3},  the problem reduces to find $t>0$ and $\xi\in\Omega$ such that
\begin{equation}\label{sistema_definitivo}
	\begin{cases}
		  c \tau(\xi) -t  V(\xi) + o(1) = 0 \ ,\\
	c \nabla  \tau(\xi) -t \nabla V(\xi)  +o(1)=0 .
	\end{cases}
\end{equation}
with $  c:=\frac{	\mathfrak C^2}{\omega \mathfrak c^2}.$
Let $\xi_0\in\Omega$ be a non-degenerate critical point of the function $f(\xi):={\tau(\xi)\over V(\xi)}$ with $V(\xi_0)>0$. 
Then the point  $(t_0,\xi_0)$, $ t_0:=c\frac{\tau(\xi_0)}{V(\xi_0)} $, is an isolated zero of the function
 $F(t,\xi):(0,+\infty)\times\Omega\to \mathbb R\times\mathbb R^4$ defined by
$$F(t,\xi):=	\Big ( c \tau(\xi) -t  V(\xi),	  c \nabla\tau(\xi) -t \nabla V(\xi)\Big).$$
We claim that the local degree
\begin{equation}\label{locdeg}\mathtt{degloc} \Big(F,(t_0,\xi_0)\Big)\not=0.\end{equation}
Indeed, 
$$ F'(t_0,x_0) = \(\begin{matrix} &-t_0  V(\xi_0)&-t_0  \nabla V(\xi_0)\\
& c \nabla\tau(\xi_0) -t_0 \nabla V(\xi_0) & c \mathcal D^2\tau(\xi_0)  -t_0 \mathcal D^2 V(\xi_0)\\
\end{matrix}\)$$
where
$$c \nabla\tau(\xi_0) -t_0 \nabla V(\xi_0)=-c\tau(\xi_0)\nabla f(\xi_0)=0,$$
so 
$$\begin{aligned}\mathtt{det}\ F'(t_0,x_0)&=c^2\frac{\tau(\xi_0)}{V(\xi_0)}\( -V(\xi_0)\mathcal D^2\tau(\xi_0) +\tau(\xi_0) \mathcal D^2 V(\xi_0) \)
\\ &= -c^2{V^3(\xi_0)\tau(\xi_0)}\mathtt{det}\ \mathcal D^2f(\xi_0)\not=0.\end{aligned}$$
Finally, \eqref{locdeg} implies that if $\epsilon$ is small enough the system \eqref{sistema_definitivo} has a solution $t_\epsilon,\xi_\epsilon$ such that $t_\epsilon\to t_0 $ and $\xi_\eps\to\xi_0$ as $\epsilon\to0.$ That concludes the proof.

\bibliography{biblio}
\bibliographystyle{abbrv}

\end{document}